\documentclass{birkmult}

\usepackage{pstricks,pst-plot}
\usepackage{color}
\usepackage{verbatim}

 \newtheorem{thm}{Theorem}[section]
 
 \newtheorem{lem}[thm]{Lemma}
 
 \theoremstyle{definition}
 
 \theoremstyle{remark}
 \newtheorem{rem}[thm]{Remark}
 
 \numberwithin{equation}{section}
 \newcommand{\F}{\mathcal{F}}

\begin{document}
\setlength{\baselineskip}{16pt}
%
%
%
%
%
%
%
%
%
\title[Oscillatory integrals and decay rates for wave-type equations]
 {Estimates for a class of oscillatory integrals and decay rates for wave-type equations}
\author{Anton Arnold}

\address{Institut f\"{u}r Analysis und Scientific Computing, \\
Technische Universit\"{a}t Wien\\
Wiedner Hauptstr. 8, A-1040 Wien, Austria;} \email{anton.arnold@tuwien.ac.at}

\thanks{This work was supported by the Postdoctoral Science Foundation of Huazhong University of Science and Technology in China and the Eurasia-Pacific Uninet scholarship for post-docs in Austria.
The first author was supported by the FWF (project I 395-N16). {The third author
was supported by NSFC (No. 10801057), the Key Project of Chinese
Ministry of Education (No. 109117), NCET-10-0431, and  CCNU Project (No. CCNU09A02015)}}

\author{JinMyong Kim}

\address[Current Address]{%
: Institut f\"{u}r Analysis und Scientific Computing, \\
Technische Universit\"{a}t Wien\\
Wiedner Hauptstr. 8, A-1040 Wien, Austria;\\[2mm]
(Permanent Address) : Department of Mathematics,\\
Kim Il Sung University\\
Pyongyang,  DPR Korea;} \email{jinjm39@yahoo.com.cn}

\author{Xiaohua Yao }
\address{ Department of Mathematics\\
Central China Normal University \\
Wuhan 430079, P. R. China;} \email{yaoxiaohua@mail.ccnu.edu.cn} 

\subjclass{ 42B20; 42B37; 35L25; 35B65}

\keywords{Oscillatory integral, higher-order wave equation,
fundamental solution estimate}

\date{January 1, 2011}
\dedicatory{}

\begin{abstract}
This paper investigates higher order wave-type equations of the form
$\partial_{tt}u+P(D_{x})u=0$, where the symbol $P(\xi)$ is a real,
non-degenerate elliptic polynomial of the order $m\ge4$ on ${\bf R}^n$.
Using  methods from harmonic analysis, 
we first establish global pointwise time-space
estimates for a class of oscillatory integrals that appear as the fundamental
solutions to the Cauchy problem of such wave equations. 
These estimates are then used 
to establish (pointwise-in-time) $L^p-L^q$ estimates on the wave solution in terms of the 
initial conditions.
\end{abstract}

\maketitle
\section{Introduction}\label{S1}

It is well known that the solution $u(t,x)$ of the Cauchy problem for the wave
equation:
$$
\left\{
\begin{array}{ll}
 \partial_{tt}u(t,x)-\triangle u(t,x)=0,\quad (t,x)\in {\bf R}\times {\bf R}^n\\
u(0,x)=u_0(x),\partial_tu(0,x)=u_1(x),\quad x\in {\bf R}^n
\end{array}
\right.
$$
has the following form:
\begin{equation} \label{eq:11}
u(t,x)={\F}^{-1}\cos(|\xi|t)\,{\F}u_0+{\F}^{-1}\frac{\sin(
|\xi|t)}{|\xi|}\,{\F}u_1,
\end{equation}
where ${\F}$ (resp. ${\F}^{-1}$) denotes the Fourier transform (resp.
its inverse). On the other hand,
\begin{equation} \label{eq:12}
u(t,x)={\F}^{-1}\cos\left([1+|\xi|^2]^{1/2}t\right)\,{\F}u_0
+{\F}^{-1}\frac{\sin\left( [1+|\xi|^2]^{1/2}t \right)}{(1+|\xi|^2)^{1/2}}\,{\F}u_1
\end{equation}
is the solution of the linear Klein-Gordon equation:
$$
\left\{
\begin{array}{ll}
    \partial_{tt}u(t,x)-\triangle u(t,x)+u(t,x)=0,\quad (t,x)\in {\bf R}\times {\bf R}^n  \\
    u(0,x)=u_0(x),\;\partial_tu(0,x)=u_1(x), \quad x\in {\bf R}^n.\\
\end{array}
\right.
$$
  If we use $P(\xi)=|\xi|^2$ and
$P=1+|\xi|^2$, respectively in \eqref{eq:11} and \eqref{eq:12}, then the above solutions read as follows:
\begin{equation} \label{eq:13}
 \begin{array}{c}
   \displaystyle{u(t,x)={\F}^{-1}\cos \left(P^{1/2}(\xi)t\right)\,{\F}u_0+{\F}^{-1} \frac {\sin
   \left(P^{1/2}(\xi)t\right)}{P^{1/2}(\xi)}\,{\F}u_1}
   \\[4mm]
   \!\!\!\! \displaystyle{=\Big({\F}^{-1} \frac{e^{iP^{1/2}(\xi)t}+e^{-iP^{1/2}(\xi)t}}{2}\Big)\ast
   u_0+\Big({\F}^{-1} \frac
   {e^{iP^{1/2}(\xi)t}-e^{-iP^{1/2}(\xi)t}}{2iP^{1/2}(\xi)}\Big)\ast u_1.}
 \end{array}
\end{equation}
The main focus of this paper is to derive pointwise
estimates (both in $t$ and $x$) on the oscillatory integrals
\begin{equation} \label{eq:14}
I_1(t,x):=\int_{{\bf R}^n} e^{i<x,\xi>\pm itP^{1/2}(\xi)}d\xi
\end{equation}
and
\begin{equation} \label{eq:15}
I_2(t,x):=\int_{{\bf R}^n} e^{i<x,\xi>\pm itP^{1/2}(\xi)}P^{-1/2}(\xi)d\xi
\end{equation}
appearing in \eqref{eq:13} -- but for a larger class of symbols $P(\xi)$.
{}From such estimates on the fundamental solution one can then derive 
solution properties, like its spatial decay at a fixed time, or decay/growth estimates of 
$\|u(t,.)\|_{L^q}$ in time.

For the classical wave and the Klein-Gordon equations, such pointwise-in-time $L^p$--$L^q$ decay estimates (i.e.\ estimates on $\|u(t,.)\|_{L^q}$ in terms of $\|u_0\|_{L^p}$ and $\|u_1\|_{L^p}$) can be found frequently in the literature \cite{Br, Mo, MSW, Mia, Miy, Str}.  It is also well known that such $L^p$--$L^{p'}$ estimates allow to deduce the famous Strichartz inequalities, which are very useful for the analysis of nonlinear wave equations (see e.g.\ \cite{GV, KT, Stra, So1, Ta}). More generally, many similar Strichartz-type estimates  (local and global in time, or with certain spatial weights) for second order hyperbolic equations have been established in the case of variable coefficients or on Riemannian manifolds. There, crucial analytic tools from microlocal analysis or spectral theory are employed (see e.g.\ \cite{Be, BTz,Bu, Ka, MSS, SS1, SS2, Tat} and the references therein). We remark that these mentioned Strichartz-type estimates are for space-time-integrals, while our estimates are all pointwise in time.

In this paper, our main aim is to derive $L^p$--$L^q$ estimates for the following general wave-type equations:
$$
\left\{
\begin{array}{ll}
 \partial_{tt}u(t,x)+P(-i\nabla)u(t,x)=0,\quad (t,x)\in {\bf R}\times {\bf R}^n\\
u(0,x)=u_0(x),\partial_tu(0,x)=u_1(x),\quad x\in {\bf R}^n\,,
\end{array}
\right.
$$
where $P(\xi)$ is a positive, real valued polynomial of higher (even) order $m\ge4$ on ${\bf R}^n$.
In order to derive $L^p$--$L^q$ estimates of the solution \eqref{eq:13}, it suffices to study pointwise estimates of  the oscillatory integrals \eqref{eq:14} and \eqref{eq:15} associated to the general polynomial $P$. To this end, we need the following assumptions on $P(\xi)$:

\vskip0.3cm
\textbf{(H1)}: $P:{\bf R}^n\rightarrow {\bf R}$ is a real elliptic inhomogeneous
polynomial of even order $m\geq 4$ with $P(\xi)>0$ for all $\xi\in{\bf R}^n$,
and $n\ge2$.

\textbf{(H2)}: $P$ is  non-degenerate, i.e. the determinant of the
Hessian
\begin{equation} \label{eq:16}
{\rm
det}\Big{(}\frac{\partial^{2}P_m(\xi)}{\partial\xi_{i}\partial\xi_{j}}\Big{)}_{n\times
n}\neq0 \quad \forall \ \xi\in {\bf R}^n\backslash\{0\}\,,
\end{equation}
where $P_m$ is the principal part of $P$. \vskip0.3cm It
is well known that for elliptic polynomials $P$, condition
\textbf{(H2)} is equivalent to the following condition
\textbf{(H2$'$)} (see Lemma 2 in \cite{C2}). \vskip0.4cm
 \textbf{(H2$'$)}: For any fixed $z\in {\bf S}^{n-1}$(the unit sphere of ${\bf R}^n$), the
function $\psi(\omega):=\,<z,\omega>(P_m(\omega))^{-1/m}$, defined on ${\bf S}^{n-1}$,
 is non-degenerate at its critical points. This means: If $ d_{\omega}\psi$, the differential of $\psi$ at a point $\omega\in {\bf S}^{n-1}$
vanishes, then $ d_{\omega}^{2}\psi$, the second order differential of $\psi$ at this point
is non-degenerate.
Note that the non-degeneracy of  $P$ is also equivalent to
det$(\partial_i\partial_j P(\xi))_{n\times n}$ being an elliptic
polynomial of order $n(m-2)$.

Particular examples of such higher order wave-type equations have already been studied in
several papers. For $P=1+|\xi|^{4}$ (linear beam equations of forth order), Levandosky \cite{Le} obtained  $L^p$--$L^q$ estimates and space-time integrability estimates. He used them to study the local existence and the asymptotic behavior of solutions to the nonlinear equation with nonlinear terms growing like a certain power of $u$. Further, Levandosky and Strauss \cite{LS}, Pausader \cite{Pa, Pa1} established the scattering theory of the nonlinear beam equation with subcritical nonlinear terms for energy initial values. Even earlier, for $P=1+|\xi|^{^m}$ (with $m\geq4$ even), Pecher \cite{P} studied $L^p$--$L^{p'}$ estimates of such higher order wave equations and also considered their application to nonlinear problems. Clearly, these polynomials are special cases satisfying our assumptions stated above. In the sequel, we shall deal with the general class in the form of the oscillatory integrals \eqref{eq:14} and \eqref{eq:15} under the assumptions \textbf{(H1)} and \textbf{(H2)}. Comparing with the classical wave equation and Klein-Gorden equations,  the fundamental solutions of higher order wave-type equations behave ``better'' in the dispersions relation and w.r.t.\ the gain of a certain decay in the space variable $x$.  As a consequence, we can obtain a larger set of admissible $(1/p,1/q)$--pairs such that the $L^p$--$L^q$ estimates hold (see \S \ref{S4}).  Concerning dispersive estimates, our methods (mainly from harmonic analysis) and results are similar to those of various dispersive Schr\"odinger-type  equations. And on this topic there exists a vast body of literature, see e.g.\ \cite{BE, BKS, BKS2,BS,BT, CMY, C1,C2, YY1,YY2, KAY, KPV, Mia,YZ,ZYF}.


\vskip0.3cm The oscillatory integrals \eqref{eq:14} and
 \eqref{eq:15} can initially be understood in the distributional
sense. Based on the assumption that $P$ is elliptic, it is easy
to see that $I_j(t,x);\,j=1,2$ are infinitely differentiable
functions in the $x$ variable for every fixed $t\neq0$ (e.g.  see \S1 of \cite{C2}).
In this paper, we shall derive pointwise time-space estimates for the oscillatory
integrals \eqref{eq:14} and \eqref{eq:15}. Subsequently, such estimates are used to establish
$L^p-L^q$ estimates for the wave solutions.  Finally, we also remark that, based on these $L^p-L^q$ estimates,
some applications to nonlinear problems can be expected, which will be investigated in a following paper.


This paper is organized as follows. In Section \ref{S2}, we make some
pretreatment to the oscillatory integrals \eqref{eq:14} and \eqref{eq:15}, review
the (polar coordinate transformation) method of Balabane et al.\ \cite{BE} and its extension by Cui \cite{C2}.  
In the core Section \ref{S3} we prove the pointwise
time-space estimates on \eqref{eq:14}, \eqref{eq:15}, following the strategy from \cite{KAY}. Finally, in \S 4 these estimates are applied to obtain $L^p-L^q$
estimates for solutions to higher order wave equations.

\section{Preliminaries}\label{S2}

We denote by ${\bf S}^{n-1}$ the unit sphere in ${\bf R}^n$, and by
$(\rho,\omega)\in[0,\infty) \times{\bf S}^{n-1}$ the polar
coordinates in ${\bf R}^n$.
Throughout this paper, we assume that  $P:{\bf R}^n\rightarrow {\bf R}$
satisfies the assumptions (${\bf H1}$) and (${\bf H2}$) (or
(${\bf H2'}$)).
Hence, $P_m(\xi)>0$ for $\xi\ne0$. This implies that there exists a large enough constant
$a>0$ with: For each fixed $s\ge a$ and each fixed $\omega\in{\bf
S}^{n-1}$, the equation $P(\rho\omega)=s$  has a unique positive solution $\rho=\rho(s,\omega)\in
C^\infty([a,\infty)\times{\bf S}^{n-1})$. By Lemma 2 in \cite{BE}, $\rho$ can be decomposed as
\begin{equation} \label{eq:21}
\rho(s,\omega)=s^\frac{1}{m}(P_{m}(\omega))^{-\frac{1}{m}}
+\sigma(s,\omega),
\end{equation}
where $\sigma\in S^0_{1,0}([a,\infty)\times{\bf S}^{n-1})$. This symbol class denotes functions in
$C^\infty ([a,\infty)\times{\bf S}^{n-1})$ that satisfy the following condition (cf.  \cite{C2,St}):
For every $k\in{\bf N}_0$ and every differential operator
$L_\omega$ on the sphere ${\bf S}^{n-1}$, there exists a constant $C_{kL}$ such
that
\begin{equation} \label{eq:22}
|\partial^k_s L_\omega\sigma(s,\omega)|\le C_{kL}(1+s)^{-k}
\quad{\rm for}\ s\ge a\ {\rm and}\ \omega\in{\bf S}^{n-1}.
\end{equation}
\vskip0.3cm We now recall two lemmata (see \cite{BE,C2}) for the
following phase function
$$
\phi(s,\omega):=s^{-\frac{1}{m}}\rho(s,\omega)\langle z,\omega
\rangle\quad{\rm for}\ s\ge a\ {\rm and}\ \omega\in{\bf S}^{n-1},
$$
with some fixed $z\in{\bf S}^{n-1}$.  Clearly, $\phi\in
S^0_{1,0}([a,\infty)\times{\bf S}^{n-1})$. For every fixed $z_0\in{\bf
S}^{n-1}$ there exists a (sufficiently small) neighborhood
$U_{z_0}\subset{\bf S}^{n-1}$ of $z_0$ such that the following
lemmata hold uniformly in $z\in U_{z_0}$ (i.e.\ the constants in Lemma \ref{lem:21}, Lemma \ref{lem:22},  and Lemma \ref{lem:23} are then independent of $z$). Therefore we do not write the
variable $z$ in the function $\phi$.

\begin{lem}[Lemma 4 of \cite{C2}, Lemma 3 of \cite{BE}] \label{lem:21}
There exists a constant $a_0\ge a$ and an open cover
$\{\Omega_0,\Omega_{\mbox{\tiny +}},\Omega_{\mbox{-}}\}$ of ${\bf
S}^{n-1}$ with $\Omega_{\mbox{\tiny
+}}\cap\Omega_{\mbox{-}}=\emptyset$ such that it holds for $s\ge a_0$:

(a) The function $\Omega_0\ni\omega\mapsto\phi(s,\omega)$ has no
critical points, and
\begin{equation} \label{eq:23}
\|d_{\omega}\phi(s,\omega)\|\ge c>0\quad{\rm for}\
\omega\in\Omega_0,
\end{equation}
where the constant $c$ is independent of $s$.

(b) Each of the two functions $\Omega_{\mbox{\tiny
$\pm$}}\ni\omega\mapsto\phi(s,\omega)$ has a unique critical point, which satisfies:
$\omega_{\mbox{\tiny $\pm$}}=\omega_{\mbox{\tiny $\pm$}}(s)\in C^\infty([a_0,\infty);
\Omega'_{\mbox{\tiny $\pm$}})$ for some open subset
$\Omega'_{\mbox{\tiny $\pm$}}$ with $\overline{\Omega'}_{\mbox{\tiny
$\pm$}}\subset\Omega_{\mbox{\tiny $\pm$}}$, respectively. Furthermore,
\begin{equation} \label{eq:24}
\|(d^2_{\omega}\phi(s,\omega))^{-1}\|\le c_0 \quad{\rm for}\
\omega\in\Omega_{\mbox{\tiny $\pm$}},
\end{equation}
where the constant $c_0$ is independent of $s$. Moreover,
$\lim_{s\to\infty}\omega_{\mbox{\tiny $\pm$}}(s)$ exists and
$$
|\omega_{\mbox{\tiny $\pm$}}^{(k)}(s)|\le c_k(1+s)^{-k-\frac{1}{m}}
\quad{\rm for}\ k\in{\bf N}.
$$
\end{lem}
\vskip0.3cm
\begin{lem}[Lemma 6 of \cite{C2}] \label{lem:22}
We define $\phi_{\mbox{\tiny $\pm$}}(t,r,s):=st+rs^\frac{2}{
m}\phi (s^2,\omega_{\mbox{\tiny $\pm$}}(s^2))$ for $t$, $r>0$, and
$s\ge a$. Then, there exist constants $a_1\ge \max(a_0,\sqrt a)$ and $c_2>c_1>0$ such
that we have for $s\ge a_1$, $t>0$, and $r>0$:
\begin{equation} \label{eq:25}
c_1\le\pm\phi(s,\omega_{\mbox{\tiny $\pm$}}(s))\le c_2,
\end{equation}
\begin{equation} \label{eq:26}
\partial_s\phi_{\mbox{\tiny +}}(t,r,s)\ge t+c_1rs^{\frac{2}{m}-1},
\end{equation}
\begin{equation} \label{eq:27}
t-c_2rs^{\frac{2}{m}-1}\le\partial_s\phi_{\mbox{-}}(t,r,s)\le
t-c_1rs^{\frac{2}{m}-1},
\end{equation}
\begin{equation} \label{eq:28}
c_1rs^{\frac{2}{m}-2}\le|\partial^2_s\phi_{\mbox{-}}(t,r,s)|\le
c_2rs^{\frac{2}{m}-2},
\end{equation}
and
\begin{equation} \label{eq:29}
|\partial^k_s\phi_{\mbox{\tiny $\pm$}}(t,r,s)|\le c_2rs^{\frac{2}{
m}-k}\quad{\rm for}\ k=2,3,\cdots.
\end{equation}
\end{lem}
\vskip0.4cm With this preparation we are able to estimate the following oscillatory integral
\begin{equation} \label{eq:29a}
\Phi(\lambda,s):=\int_{{\bf S}^{n-1}}e^{i\lambda\phi(s,\omega)}
b(s,\omega)d\omega,
\end{equation}
where $b(s,\omega):=s^{1-\frac{n}{m}}\rho^{n-1}\partial_s\rho\in
S^0_{1,0} ([a,\infty)\times{\bf S}^{n-1})$ and $\lambda>0$. Let
$\varphi_{\mbox{\tiny +}}$, $\varphi_{\mbox{-}}$, $\varphi_0$ be a
partition of unity of ${\bf S}^{n-1}$, subordinate to the open
cover given in Lemma \ref{lem:21}. 
Then we decompose $\Phi$ as
$$
\Phi(\lambda,s)=\Phi_{\mbox{\tiny +}}(\lambda,s)+\Phi_{\mbox{-}}
(\lambda,s)+\Psi_0(\lambda,s),
$$
where
$$
\Phi_{\mbox{\tiny $\pm$}}(\lambda,s):=\int_{{\bf
S}^{n-1}}e^{i\lambda\phi(s,\omega)} b(s,\omega)\varphi_{\mbox{\tiny
$\pm$}}(\omega)d\omega
$$
and
$$
\Psi_0(\lambda,s):=\int_{{\bf S}^{n-1}}e^{i\lambda\phi(s,\omega)}
b(s,\omega)\varphi_0(\omega)d\omega.
$$
By using the stationary phase method for $\Psi_0$, and Lemma \ref{lem:21} and
\cite{So} (Corollary 1.1.8, \S1.2) for $\Phi_{\mbox{\tiny $\pm$}}$, one obtains the
following result.

\begin{lem} \label{lem:23}
For $\lambda>0$ and $s>a_1$ we have
\begin{equation} \label{eq:210}
\Phi(\lambda,s)=\lambda^{-\frac{n-1}{2}}e^{i\lambda
\phi(s,\omega_{\mbox{\tiny +}}(s))}\Psi_{\mbox{\tiny +}}(\lambda,s)
+\lambda^{-\frac{n-1}{2}}
e^{i\lambda\phi(s,\omega_{\mbox{-}}(s))}\Psi_{\mbox{-}}(\lambda,s)+\Psi_0(\lambda,s),
\end{equation}
where $\Psi_{\mbox{\tiny $\pm$}}$, $\Psi_0\in
C^\infty((0,\infty)\times[a_0,\infty))$ and
\begin{equation} \label{eq:211}
|\partial^k_\lambda\partial^j_s\Psi_{\mbox{\tiny
$\pm$}}(\lambda,s)|\le c_{k,j}(1+\lambda)^{-k}s^{-j}\quad {\rm for}\
k,j\in{\bf N}_0,
\end{equation}
\begin{equation} \label{eq:212}
|\partial^k_\lambda\partial^j_s\Psi_0(\lambda,s)|\le c_{k,j,l}
(1+\lambda)^{-l}s^{-j}\quad{\rm for}\ k,j,l\in{\bf N}_0.
\end{equation}
\end{lem}

\section{Estimates on the oscillatory integrals}\label{S3}

In this section we establish 
pointwise time-space
estimates of the oscillatory integrals \eqref{eq:14} and \eqref{eq:15}.
Like in \cite{KAY}, we aim at simultaneous estimates in the time and spatial variables.
This is a refinement of the analysis in \cite{C2}, where only spatial decay estimates of the
oscillatory integrals are derived. With our refined analysis we are able to give here global-in-time estimates on the wave solution.

\begin{thm}\label{th3.1}
Assume that the polynomial $P$ satisfies the conditions
{\bf (H1)} and {\bf (H2)} from \S\ref{S1}, and let $n\ge m$. Then there exists a
constant $C>0$ such that

\begin{equation} \label{eq:31}
|I_2(t,x)|\le
\left\{
\begin{array}{ll}
    C|t|^{-\frac{n-m_1}{m_1}} (1+|t|^{-{\frac{1}{m_1}}}|x|)^{-\mu}, \ {\rm for}\ 0<|t|\leq1, \\
    C|t|^{-\frac1m} (1+|t|^{-1}|x|)^{-\mu}, \ {\rm for} \ |t|\ge 1,\\
\end{array}
\right.
\end{equation}
where $m_1:=\frac{m}{2}$, $\mu:=\frac{mn-4n+2m}{2(m-2)}>0$.
\end{thm}

\begin{proof} In the sequel,  $C$ denotes some generic (but not necessarily identical) positive constants,
independent of $t$, $\xi$, $x$, and so forth. Since the integrals
$I_2(t,x)$ and $I_2(-t,x)$ are structurally identical, it suffices to
estimate $I_2(t,x)$ for $t>0$. We shall now analyze $I_2$ for three different
cases of its arguments, starting with the most delicate situation.

\noindent
\underline{ Case (i):} $t\ge1$ and $r:=|x|\ge t$.

Choose $\psi\in C^\infty({\bf R})$ such that
$$ \psi(s)=
\left\{
\begin{array}{ll}
    0, \ {\rm for}\ s\le a_1 \\
   1, \ {\rm for}\ s>2a_1,\\
\end{array}
\right. $$
 where $a_1$ is given in Lemma \ref{lem:22}. We write
\begin{eqnarray*}
I_2(t,x)&=&\int_{{\bf R}^n}e^{i(\langle x,\xi\rangle\pm
tP^{1/2}(\xi))}P^{-\frac{1}{2}}(\xi)\psi(P^{1/2}(\xi))d\xi\\
&+&\int_{{\bf R}^n} e^{i(\langle x,\xi\rangle\pm
tP^{1/2}(\xi))}P^{-\frac{1}{2}}(\xi)[1-\psi(P^{1/2}(\xi))]d\xi\\
&=:&I_{21}(t,x)+I_{22}(t,x).
\end{eqnarray*}
First we rewrite $I_{22}$ as the Fourier transform of a measure, supported on the graph $S:=\{z=\pm P^{1/2}(\xi);\:\xi\in {\bf R}^n\}\subset {\bf R}^{n+1}$:
\begin{equation}\label{surface-meas}
  I_{22}(t,x)=\int_{{\bf R}^{n+1}}e^{i(\langle x,\xi\rangle + tz)} P^{-\frac{1}{2}}(\xi)[1-\psi(P^{1/2}(\xi))] \delta(z\mp P(\xi)^{1/2})\,d\xi\,dz\,.
\end{equation}
Since the polynomial $P$ is of order $m$, the supporting manifold of the above integrand is of type less or equal $m$ (in the sense of \S~VIII.3.2, \cite{St}; see the Appendix in \S\ref{S:app}). Then, Theorem 2 of \S~VIII.3 in \cite{St} implies
\begin{equation} \label{eq:310a}
  |I_{22}(t,x)|\le C(1+|t|+|x|)^{-\frac1m}\qquad \forall t,\,x.
\end{equation}
This can be generalized: 
Since $f(t,\xi):=e^{\pm itP^{1/2}}P^{-1/2}[1-\psi(P^{1/2})]\in C_c^\infty({\bf
R}^n)$ for every $t>0$, we obtain by integration by parts
$$
  I_{22}(t,x)=i \int_{{\bf R}^n}e^{i\langle x,\xi\rangle}
  \frac{x}{|x|^2}\cdot \nabla_\xi f(t,\xi)d\xi\,.
$$
Proceeding recursively this implies (in the spirit of the Paley-Wiener-Schwartz theorem)
\begin{equation} \label{eq:31a}
|I_{22}(t,x)|\le C_kt^kr^{-k}\quad{\rm for}\ k\in{\bf N}_0,\,x\ne0,\,t\ge1,
\end{equation}
and hence also $\forall\,k\ge0$.
But proceeding as in \eqref{surface-meas} yields the improvement
\begin{equation} \label{eq:32}
  |I_{22}(t,x)|\le C_k |t|^{-\frac1m}(1+|t|^{-1}|x|)^{-(k+\frac1m)}\ {\rm for}\ \
  |t|\ge 1, \ x\in{\bf R}^n, \ \forall \ k\ge0.
\end{equation}

To estimate $I_{21}$, we shall derive an $\varepsilon$--uniform estimate of its regularization
\begin{equation} \label{eq:32b}
J_\varepsilon(t,x):=\int_{{\bf R}^n}e^{-\varepsilon
P^{1/2}(\xi)+i(\langle x,\xi\rangle\pm
tP^{1/2}(\xi))}P^{-1/2}(\xi)\psi(P^{1/2}(\xi))d\xi\quad{\rm for}\
\varepsilon>0.
\end{equation}
By the polar coordinate transform and the change of variables
$(\rho,\omega)\to(s,\omega)$ such that
$\rho=\rho(s,\omega)$ (with $P(\rho\omega)=s$),  we have
\begin{eqnarray} \label{eq:32a}
J_\varepsilon(t,x)\!\!&\!\!\!=\!\!\!\!&\int^\infty_0\!\!\!\int_{{\bf
S}^{n-1}}\!\!\!e^{-\varepsilon P^{1/2}(\rho\omega)+i(\rho\langle
x,\omega\rangle\pm tP^{1/2}(\rho\omega))}P^{-1/2}(\rho\omega)
\psi(P^{1/2}(\rho\omega))\rho^{n-1}d\omega d\rho\nonumber \\
&\!\!\!=\!\!\!\!&\int^\infty_0\!\!\!\int_{{\bf S}^{n-1}}e^{-\varepsilon \sqrt{s}\pm
it\sqrt{s}+
ir\rho\langle z,\omega\rangle}\psi(\sqrt{s})s^{-1/2}\rho^{n-1}\partial_s\rho d\omega ds\\
&\!\!\!=\!\!\!\!&\int^\infty_0 e^{-\varepsilon \sqrt{s}\pm it\sqrt{s}}s^{\frac{n}{m}-1}s^{-1/2}\psi(\sqrt{s})\Phi(rs^\frac{1}{m},s)ds\nonumber \\
&\!\!\!=\!\!\!\!&2\int^\infty_0
e^{-\varepsilon s\pm its}s^{\frac{2n}{
m}-2}\psi(s)\Phi(rs^\frac{2}{m},s^2)ds\,,\nonumber
\end{eqnarray}
where $z:=x/|x|$ enters in the oscillatory integral $\Phi$ from \eqref{eq:29a}.
For the transformation $\rho\to s$ we used that $\psi(s)=0$ on $[0,a_1]$ (see \S2 in \cite{C2} for a more detailed discussion). Here and in the sequel we assume that the functions $\Phi,\,\Phi_\pm,\,\phi_\pm,\,\Psi_\pm,\,\Psi_0$ are smoothly extended to $[0,a]$, in order to write the $s$--integrals on ${\bf R}_+$. The precise form of this extension, however, will not matter -- due to the cut-off function $\psi$.

The main goal of this proof is to derive, for any $z_0\in{\bf S}^{n-1}$,
an $\varepsilon$--uniform estimate of the form $|J_\varepsilon(t,x)|\le C t^{-\nu} r^{-\mu}$, with
$\nu:=\frac{n-m}{m-2}\ge0$ (since $ n\ge m)$.
Because of the Lemmata \ref{lem:21}--\ref{lem:23}, this estimate will hold uniformly on
$z=x/|x|\in U_{z_0}$ with a constant $C=C(z_0)$.
Due to the compactness of ${\bf S}^{n-1}$, finitely many points $z_1,...,z_N$ will suffice to yield  a uniform estimate of $|J_\varepsilon(t,x)|$ on $\{r\ge t\ge1\}$, using
$\displaystyle C=\max_{j=1,...,N} C(z_j)$.
Here, we only consider the case of
$e^{-\varepsilon s+its}$; for $e^{-\varepsilon s-its}$ the estimates
are analogous.

Following Lemma \ref{lem:23} we decompose $J_\varepsilon$ as follows:
\begin{eqnarray*}
J_\varepsilon(t,x)&=&2r^{-\frac{n-1}{2}}\int_0^\infty e^{-\varepsilon
s+i\phi_{\mbox{\tiny +}}(t,r,s)}
s^{\frac{n+1}{m}-2}\psi(s)\Psi_{\mbox{\tiny +}}(rs^{\frac{2}{m}},s^2)ds\\
&&+2r^{-\frac{n-1}{2}}\int_0^\infty e^{-\varepsilon
s+i\phi_{\mbox{-}}(t,r,s)}
s^{\frac{n+1}{m}-2}\psi(s)\Psi_{\mbox{-}}(rs^{\frac{2}{m}},s^2)ds\\
&&+2\int^\infty_0e^{-\varepsilon s+its}s^{\frac{2n}{m}-2}\psi(s)
\Psi_0(rs^{\frac{2}{m}},s^2)ds\\
&=:&R_\varepsilon^{\mbox{\tiny
+}}(t,x)+R_\varepsilon^{\mbox{-}}(t,x)+R_\varepsilon^0(t,x)\,,
\end{eqnarray*}
where $\phi_{\mbox{\tiny $\pm$}}$ is defined in Lemma \ref{lem:22}.

We shall
first estimate the integral $R_\varepsilon^0(t,x)$ and set
$v_0(s):=s^{\frac{2n}{m}-2}\psi(s)\Psi_0(rs^{\frac{2}{m}},s^2)$. By the
Leibniz rule and \eqref{eq:212}, we have
$$
|v_0^{(k)}(s)|\le C(rs^\frac{2}{m})^{-l}s^{\frac{2n}{m}-2-k}\quad{\rm
for}\ l,k\in{\bf N}_0,
$$
where $r\ge1$ and $s\ge a_1$. Choose $l\ge\mu\ge0$ and
$k\ge\nu\ge0$. It thus follows by integration by parts that
\begin{equation} \label{eq:33}
|R^0_{\varepsilon}(t,x)|\le Ct^{-k}\int^\infty_{a_1}(rs^\frac{2}{
m})^{-l}s^{\frac{2n}{m}-2-k}ds\le Ct^{-k}r^{-l}\le Ct^{-\nu}r^{-\mu}.
\end{equation}

 To estimate the integral $R_\varepsilon^{\mbox{\tiny +}}(t,x)$, for
given $r\ge t\ge1$,  we set
$$
\left\{
\begin{array}{ll}
   u_{\mbox{\tiny +}}(s):=-\varepsilon s+i\phi_{\mbox{\tiny +}}(t,r,s)\,, \\
   v_{\mbox{\tiny +}}(s):=s^{\frac{n+1}{m}-2}\psi(s)\Psi_{\mbox{\tiny +}}(rs^\frac{2}{m},s^2)\\
\end{array}
\right.
$$
for $s\ge 0$. Since $u_{\mbox{\tiny +}}'(s)\ne0$ for $s\ge a_1$,
we can define $D_\ast f:=(gf)'$ for $f\in C^1(0,\infty)$, where
$g:=-1/u_{\mbox{\tiny +}}'$. It is not hard to show
\begin{equation} \label{eq:34}
D^j_\ast v_{\mbox{\tiny +}}=\sum_\alpha c_\alpha
g^{(\alpha_1)}\cdots g^{(\alpha_j)}v_{\mbox{\tiny
+}}^{(\alpha_{j+1})} \quad{\rm for}\ j\in{\bf N} \,,
\end{equation}
where
the sum runs over all $\alpha=(\alpha_1,\cdots, \alpha_{j+1})\in{\bf
N}^{j+1}_0$ such that $|\alpha|=j$ and
$0\le\alpha_1\le\cdots\le\alpha_j$. Since \eqref{eq:26} and \eqref{eq:29} imply,
respectively, $|g(s)|\le Cr^{-1}s^{1-\frac{2}{m}}$ and
$$
|u_{\mbox{\tiny +}}^{(k)}(s)|\le Crs^{\frac{2}{m}-k}\quad{\rm for}\
k=2,3,\cdots,
$$
we find by induction on $k$:
$$
|g^{(k)}(s)|\le Cr^{-1}s^{1-\frac{2}{m}-k}\quad{\rm for}\ k\in{\bf
N}_0,
$$
which shall yield the spatial decay of $I_2$. To derive the time decay of $I_2$, we note that
\eqref{eq:26} also implies $|g(s)|\le t^{-1}$. Using this inequality for just one factor in $g^{(k)}$ we obtain:
$$
|g^{(k)}(s)|\le Ct^{-1}s^{-k}\quad{\rm for}\ k\in{\bf N}_0\,.
$$
The novel key step is now to interpolate these two inequalities, which will allow us to derive estimates also for large time.
We have for any $\theta\in[0,1]$:
\begin{equation} \label{eq:35}
|g^{(k)}(s)|\le Ct^{\theta-1}r^{-\theta}s^{\theta(1-\frac{2}{
m})-k}\quad{\rm for}\ k\in{\bf N}_0.
\end{equation}
On the other hand we have by the Leibniz rule and \eqref{eq:211}:
\begin{equation} \label{eq:36}
|v_{\mbox{\tiny +}}^{(k)}(s)|\le Cs^{\frac{n+1}{m}-2-k}\quad{\rm
for}\ k\in{\bf N}_0.
\end{equation}
It thus follows from \eqref{eq:34} -- \eqref{eq:36} that
\begin{equation} \label{eq:37}
|D^j_\ast v_{\mbox{\tiny +}}(s)|\le Ct^{j(\theta-1)}
r^{-j\theta}s^{j\theta(1-\frac{2}{m})+\frac{n+1}{m}-2-j}\quad{\rm
for}\ j\in{\bf N}_0,
\end{equation}
where $D^0_\ast v_{\mbox{\tiny +}}:=v_{\mbox{\tiny +}}$. The particular choice
$\theta=\frac{\mu}{n}$, $j=n$ yields
\begin{equation} \label{eq:37a}
|D^n_\ast v_{\mbox{\tiny +}}(s)|\le Ct^{\mu-n}
r^{-\mu}s^{\frac{-mn-2n+2}{2m}-1}.
\end{equation}
Noting that $\mu-n<-\nu$, one gets by integration by parts
$$
|R^{\mbox{\tiny +}}_\varepsilon(t,x)|=2r^{-\frac{n-1}{
2}}\Big{|}\int_0^\infty e^{u_{\mbox{\tiny +}}}(D^n_\ast
v_{\mbox{\tiny +}})ds\Big{|} \le Ct^{\mu-n}r^{-\frac{n-1}{2}-\mu}\le
Ct^{-\nu}r^{-\mu}.\\[2mm]
$$

We now turn to the integral $R_\varepsilon^{\mbox{-}}(t,x)$.
Here we put
$$
\left\{
\begin{array}{ll}
   u_{\mbox{-}}(s):=-\varepsilon s+i\phi_{\mbox{-}}(t,r,s)\,, \\
   v_{\mbox{-}}(s):=s^{\frac{n+1}{m}-2}\psi(s)\Psi_{\mbox{-}}(rs^\frac{2}{m},s^2)\\
\end{array}
\right.
$$
for $s\ge 0$.
We shall denote $s_0:=(r/t)^\frac{m}{m-2}$, $c'_1:=(c_1/2)^\frac{m}{m-2}$, and
$c'_2:=(2c_2)^\frac{m}{m-2}$, with $c_1$ and $c_2$ given in Lemma \ref{lem:22}.
Now we decompose $R_\varepsilon^{\mbox{-}}$ as
\begin{eqnarray*}
R_\varepsilon^{\mbox{-}}(t,x)&=&2r^{-\frac{n-1}{
2}}\Big{\{}\int_0^{c'_1s_0}
+\int_{c'_1s_0}^{c'_2s_0}+\int_{c'_2s_0}^\infty\Big{\}}
e^{u_{\mbox{-}}(s)}v_{\mbox{-}}(s)ds\\
&=:&R_{\varepsilon 1}^{\mbox{-}}(t,x)+R_{\varepsilon
2}^{\mbox{-}}(t,x)+R_{\varepsilon 3}^{\mbox{-}}(t,x)\,.
\end{eqnarray*}
This decomposition is motivated by the fact that the phase $\partial_s \phi_{\mbox{-}}(t,r,\cdot)$
is negative on $[0,c'_1s_0)$, positive on $[c'_2s_0,\infty)$, and is has exactly one zero on
$[c'_1s_0,c'_2s_0]$ (cf.\ \eqref{eq:27}, \eqref{eq:28}).

Integrating by parts we obtain
$$
R^{\mbox{-}}_{\varepsilon 3}(t,x)=2r^{-\frac{n-1}{
2}}\Big{(}\frac{e^{u_{\mbox{-}}(c'_2s_0)}}{u_{\mbox{-}}'(c'_2s_0)}
\sum^{n-1}_{j=0}(D^j_\ast
v_{\mbox{-}})(c'_2s_0)+\int_{c'_2s_0}^\infty
e^{u_{\mbox{-}}}(D^n_\ast v_{\mbox{-}})ds\Big{)}.
$$
Here and in the sequel, the differential operator $D_*f=(gf)'$ is considered with $g=-1/u_{\mbox{-}}'$.
Since \eqref{eq:27} implies $|u_{\mbox{-}}'(s)|\ge c_2rs^{\frac{2}{
m}-1}$ for $s\ge c'_2s_0$, we find that $v_{\mbox{-}}(s)$ also
{satisfies (the analogues of) \eqref{eq:37} and \eqref{eq:37a} 
for $s\ge c'_2s_0$. }
If $c'_2s_0\le a_1$, then $(D^j_\ast
v_{\mbox{-}})(c'_2s_0)=0$ for $j=0,\cdots ,n-1$ {(note that $\psi\equiv 0$ on $[0,a_1]$).
Integration by parts then yields}
\begin{eqnarray*}
|R^{\mbox{-}}_{\varepsilon 3}(t,x)|=\big|2r^{-\frac{n-1}{
2}}\int_{a_1}^\infty
e^{u_{\mbox{-}}}(D^n_\ast v_{\mbox{-}})ds\big|\le Ct^{-\nu}r^{-\mu},
\end{eqnarray*}
exactly as done for $R^{\mbox{\tiny +}}_\varepsilon(t,x)$.
If $c'_2s_0> a_1$, 
then
\begin{eqnarray*}
|R^{\mbox{-}}_{\varepsilon 3}(t,x)|&\!\le\!& Cr^{-\frac{n-1}{2}}\Big{(}
(rs_0^{\frac{2}{m}-1})^{-1}\sum_{j=0}^{n-1}r^{-j}s_0^{-\frac{2j-n-1}{
m}-2}+\int_{c'_2s_0}^\infty r^{-n}s^{-\frac{n+m-1}{m}-1}ds\Big{)} \\
&\!\le\!&Cr^{-\frac{n-1}{2}}(r^{-1}s_0^\frac{n-m-1}{m}\sum_{j=0}^{n-1}
(rs_0^\frac{2}{m})^{-j}+r^{-n} s_0^{-\frac{n+m-1}{m}}).
\end{eqnarray*}
Noting that $r\ge1$, $s_0> a_1/c'_2$, and $t\ge1$,  it
follows that
$$
|R^{\mbox{-}}_{\varepsilon 3}(t,x)|\le
Cr^{-\frac{n+1}{2}}s_0^\frac{n-m-1}{m}
= Ct^{-\nu}r^{-\mu} s_0^{-\frac12}t^{-\frac12}
\le Ct^{-\nu}r^{-\mu}.
$$

Next we turn to $R^{\mbox{-}}_{\varepsilon 1}(t,x)$, which is 0 for $c'_1s_0< a_1$.
If $c'_1s_0\ge a_1$, we use $|u_{\mbox{-}}'(s)|\ge\frac{1}{2}c_1rs^{\frac{2}{m}-1}$ for
$a_1\le s\le c'_1s_0$. Then, a slight modification of the above method
yields again $R^{\mbox{-}}_{\varepsilon 1}(t,x)\le Ct^{-\nu}r^{-\mu}$.

To estimate $R_{\varepsilon 2}^{\mbox{-}}(t,x)$,  it suffices to
estimate the integral
\begin{eqnarray*}
R_{0\,2}^{\mbox{-}}(t,x)
&=&2r^{-\frac{n-1}{2}}\int_{c'_1s_0}^{c'_2s_0}
e^{i\phi_{\mbox{-}}(t,r,s)}v_{\mbox{-}}(s)ds\\
&=&2r^{-\frac{n-1}{2}}s_0\int_{c'_1}^{c'_2}e^{i\phi_{\mbox{-}}
(t,r,s_0\tau)}v_{\mbox{-}}(s_0\tau)d\tau\,,
\end{eqnarray*}
where the interval of integration is now independent of the parameters $t,\,r$.
We obtain from \eqref{eq:28} that
$$
|\partial_\tau^2\phi_{\mbox{-}}(t,r,s_0\tau)|\ge
c_1rs_0^2(s_0\tau)^{\frac{2}{m}-2} \ge Crs_0^\frac{2}{m}
$$
for $\tau\in [c'_1,c'_2]$. Since $v_{\mbox{-}}(s)$ also satisfies (the analogue of)
\eqref{eq:36}, we obtain by using (a corollary of) the Van der Corput lemma (cf. \cite{St}, p.\ 334)
(uniformly for $\varepsilon>0$ small enough):
\begin{eqnarray*}
|R_{0\,2}^{\mbox{-}}(t,x)|
&\le& Cr^{-\frac{n-1}{2}}s_0(rs_0^\frac{2}{
m}) ^{-\frac{1}{2}}\Big{(}|v_{\mbox{-}}(c'_2s_0)|
+\int_{c'_1}^{c'_2}|s_0v_{\mbox{-}}'(s_0\tau)|d\tau\Big{)}\\
&\le&Cr^{-\frac{n-1}{2}}s_0(rs_0^\frac{2}{m})
^{-\frac{1}{2}}s_0^{\frac{n+1}{m}-2}\\
&=&C\,t^{-\nu}r^{-\mu}.
\end{eqnarray*}

The dominated convergence theorem implies that
$J_\varepsilon(t,\cdot)$ converges (as $\varepsilon\to0$) uniformly for
$x$ in compact subsets of $\{x\in{\bf R}^n;\ |x|\ge 1\}$. By
summarizing the above estimates we have
$$
|I_{21}(t,x)|\le Ct^{-\nu}|x|^{-\mu}\quad{\rm for}\ 
|x|\ge t\ge1,
$$
and hence
$$
|I_{21}(t,x)| 
\leq Ct^{-\frac{n}{2}}(1+t^{-1}|x|)^{-\mu}\leq Ct^{-\frac{1}{m}}(1+t^{-1}|x|)^{-\mu}
\quad{\rm for}\ |x|\ge t\ge1.
$$
Combining this with the estimate \eqref{eq:32} on $I_{22}$ (put $k=\mu-\frac1m$), we have
$$
|I_2(t,x)|\le Ct^{-\frac{1}{m}}(1+t^{-1}|x|)^{-\mu} \quad {\rm for}\ |x|\ge t\ge1.
$$

\vskip 0.3cm
\noindent
\underline{ Case (ii):} $t\geq 1$ and $|x|\leq t$.\\
For $I_{21}$ we shall prove now that
\begin{equation}\label{I1-decay}
|I_{21}(t,x)|\le C |t|^{-n/2} \quad{\rm for}\ |t|\ge1\ {\rm and}\ |x|\le |t|.
\end{equation}
We proceed as in \cite{KAY} and write the integral $I_{21}(t,x)$ as follows:
\begin{eqnarray*}
I_{21}(t,x)&=&\int_{{\bf R}^n}e^{it(\pm P^{1/2}(\xi)+\langle x/t,\xi\rangle)}
P^{-1/2}(\xi)\psi(P^{1/2}(\xi))d\xi\\
&=:&\int_{{\bf R}^n}e^{it\Phi(\xi, x,t)}P^{-1/2}(\xi)\psi(P^{1/2}(\xi))d\xi,
\end{eqnarray*}
but we shall focus on the case $\Phi=P^{1/2}(\xi)+\langle x/t,\xi\rangle$, and the other case is analogous.

Since  $|x/t|\le 1$, $P^{1/2}(\xi)\le c_1|\xi|^{m_1}$, and $|\nabla P(\xi)|\ge c_2|\xi|^{m-1}$ for large $|\xi|$,   the possible critical points satisfying
$$
  \nabla_\xi\Phi(\xi, x,t)=\frac{\nabla P(\xi)}{2P^{1/2}(\xi)}+\frac{x}{t}=0
$$
must be located in some bounded ball.   In order to apply later the stationary phase principle, let $\Omega\subset {\bf R}^n$ be some open set such that ${\rm supp}\, \psi(P^{1/2})\subset \Omega$ and $|\nabla P^{1/2}(\xi)|\ge c|\xi|^{m_1-1}$ on $\Omega$.
Note that the constant $a_1$ (from the definition of $\psi$ and Lemma \ref{lem:22}) could be increased, if necessary, such that both of those conditions can hold.
Then  we decompose $\Omega$ into $\Omega_1\cup\Omega_2$, where
 $$\Omega_1=\left\{\xi\in \Omega\,;\ \ |\nabla P^{1/2}(\xi)+\frac{x}{t}|<\frac{1}{2}|\nabla P^{1/2}(\xi)|+1\right\}$$
 and
 $$\Omega_2=\left\{\xi\in \Omega\,;\ \ |\nabla P^{1/2}(\xi)+\frac{x}{t}|>\frac{1}{4}|\nabla P^{1/2}(\xi)|\right\}.$$
Since $|\frac{x}{t}|\le 1$ and $|\nabla P^{1/2}(\xi)|\rightarrow\infty $ as $|\xi|\rightarrow\infty$,  $\Omega_1$ must be a bounded domain and includes all critical points of $\Phi$ inside $\Omega$.  Now we choose smooth functions $\eta_1(\xi)$ and $\eta_2(\xi)$ such that supp $\eta_j\subset \Omega_j$ and $\eta_1(\xi)+\eta_2(\xi)=1$ in $\Omega$. And  we decompose $I_{21}$ as
\begin{eqnarray*}
I_{21}(t,x)&=&I_{211}(t,x)+I_{212}(t,x),\\ 
I_{21j}(t,x) 
&:=&\int_{{\bf R}^n}e^{it\Phi(\xi, x,t)}\eta_j(\xi)P^{-1/2}(\xi)\psi(P^{1/2}(\xi))d\xi; \ j=1,2.
\end{eqnarray*}
\vskip 0.2cm

To estimate $I_{211}$ we note that 
$$\mbox{det}(\partial_{\xi_i}\partial_{\xi_j}\Phi)_{n\times n}(\xi,x,t)
=\mbox{det}(\partial_{\xi_i}\partial_{\xi_j}P^{1/2})_{n\times n}(\xi)\,.
$$
Lemma \ref{lem:53} (see the Appendix below) implies that the r.h.s.\ is nonzero on $\Omega$ (if necessary, we can increase the value of $a_1$ to satisfy the requirement), that is,  the Hessian matrix is non-degenerate on $\Omega$.  Moreover, $|\partial_\xi^{\alpha}\Phi|\le C_\alpha$ on $\Omega_1$ for any multi-index $\alpha \in {\bf N}_0^n$.  Hence we obtain  by the stationary phase principle  that
\begin{equation} \label{eq:38a}
|I_{211}(t,x)|\le C|t|^{-n/2}.
\end{equation}

To estimate  $I_{212}$, we shall use some cut-off in order to make the subsequent integrations by parts meaningful (cp.\ to the procedure in \eqref{eq:32b}). Using a smooth, compactly supported cut-off function $0\le\varphi\le1$ with $\varphi(0)=1$, we shall derive an $\varepsilon$--uniform estimate (as $\varepsilon\to0$) of 
$$
  I_{212}^\varepsilon(t,x) 
  :=\int_{{\bf R}^n}e^{it\Phi(\xi, x,t)}\eta_2(\xi)\varphi(\varepsilon\xi)P^{-1/2}(\xi)\psi(P^{1/2}(\xi))d\xi\,.
$$
Note that $|\nabla_\xi\Phi|=|\nabla P^{1/2}(\xi)+\frac{x}{t}|\ge \frac{1}{4}|\nabla P^{1/2}(\xi)|\ge c|\xi|^{m_1-1}$ for $\xi\in \Omega_2$ and $|\partial_\xi^{\alpha}\Phi|\le C_\alpha |\xi|^{m_1-\alpha}$ for $|\alpha|\ge2$. Now we define the operator $L$ by
$$Lf:=\frac{\langle\nabla_\xi\Phi,\nabla_\xi\rangle}{it|\nabla_\xi\Phi|^2}f.
$$
Since $Le^{it\Phi}=e^{it\Phi}$, we obtain by $N$ iterated integrations by parts:
\begin{eqnarray}\label{eq:38b}
|I_{212}^\varepsilon(t,x)|&=&\left|\int_{{\bf R}^n}e^{it\Phi(\xi, x,t)}(L^*)^N\,\left[\varphi(\varepsilon\xi)
\eta_2(\xi)P^{-1/2}(\xi)\psi(P^{1/2}(\xi))\right]d\xi\right|\nonumber \\
&\le& C_N|t|^{-N}\int_{{\rm supp}\psi(P^{1/2})}|\xi
|^{-mN}d\xi\le C'_N|t|^{-N},
\end{eqnarray}
where $N>n$ and $L^*$ is the adjoint operator of $L$.
Combining the estimates \eqref{eq:38a} and \eqref{eq:38b} yields the claimed estimate 
$|I_{21}|\le C|t|^{-n/2}$ for $|t|\ge1$ and $|x|\le |t|$.\\

Together with the estimate \eqref{eq:32} (with $k=\mu-\frac1m$) on $I_{22}$ this yields
\begin{equation} 
|I_2(t,x)|\leq Ct^{-\frac1m}(1+t^{-1}|x|)^{-\mu}\quad{\rm for}\ t\ge1\ {\rm and}\ |x|\le |t|.
\end{equation}
Thus, combining the cases (i) and (ii),  we conclude
\begin{equation} \label{eq:38}
|I_2(t,x)|\le Ct^{-\frac1m}(1+t^{-1}|x|)^{-\mu}\quad{\rm for}\ t\ge1\ {\rm and}\
x\in{\bf R}^n.
\end{equation}
\vskip0.3cm

\vskip0.3cm
\noindent
\underline{ Case (iii):} For $0<t<1$ and $x\in{\bf R}^n$ we shall use a standard scaling argument.\\
We observe that
\begin{eqnarray*}
I_2(t,x)&=&\int_{{\bf R}^n}e^{i(\langle
x,\xi\rangle+tP^{1/2}(\xi))}P^{-1/2}(\xi)d\xi \\
&=&t^{-\frac{n}{m_1}}\int_{{\bf R}^n} e^{i\big(\langle t^{-\frac{1}{
m_1}}x,\xi\rangle+tP^{1/2}(t^{-\frac{1}{
m_1}}\xi)\big)}P^{-1/2}(t^{-1/m_1}\xi)d\xi\\
&=&t^{-{\frac{n}{m_1}+1}}\int_{{\bf R}^n} e^{i\big(\langle t^{-\frac{1}{
m_1}}x,\xi\rangle+(t^2P((t^2)^{-\frac{1}{
m}}\xi))^{1/2}\big)}(t^2P((t^2)^{-\frac{1}{m}}\xi))^{-1/2}d\xi .
\end{eqnarray*}
Let $P_t(\xi):=t^2P(t^{-\frac{2}{m}}\xi)$, $\rho_t(s,\omega):=t^\frac{2}{m}\rho(\frac{s}{t^2},\omega)$, and $\sigma_t(s,\omega):=t^\frac{2}{m}\sigma(\frac{s}{t^2},\omega)$. Then \eqref{eq:21} still holds when $P$,
$\rho$, $\sigma$ are replaced, respectively, by $P_t$, $\rho_t$,
$\sigma_t$. It is easy to check that $\sigma_t$ also satisfies
\eqref{eq:22} with the same constants $C_{kL}$. Hence, we can deduce from \eqref{eq:38} (with
$t=1$) that
\begin{equation} \label{eq:39}
|I_2(t,x)|\le Ct^{-\frac{{n-m_1}}{m_1}}(1+t^{-\frac{1}{
m_1}}|x|)^{-\mu},  
\quad{\rm for}\ t\in(0,1)\ {\rm and}\ x\in{\bf R}^n.
\end{equation}
This completes the proof of the theorem.
\end{proof}

\vskip0.3cm
\begin{rem}  \label{rem:32}
If one checks the details of the proof for
the cases (i) and (ii) above, one finds that the estimate
for $I_2(t=1,x)$ does not use the condition
$n\ge m$. Therefore  the estimate \eqref{eq:39} of $I_2(t,x)$ for $0<t<1$ is also obtained
by scaling without the restriction $n\ge m$.

\end{rem} Similarly to the above proof of
$I_2$,  we obtain the following
result for the oscillatory integral $I_1(x,t)$. 
\vskip0.3cm
\begin{thm}\label{th3.3}
Assume that the polynomial $P$ satisfies 
{\bf (H1)} and {\bf (H2)}. Then
\begin{equation} \label{eq:310}
|I_1(t,x)|\le
\left\{
\begin{array}{ll}
    C|t|^{-\frac{n}{m_1}} (1+|t|^{-\frac{1}{m_1}}|x|)^{-\frac{n(m-4)}{2(m-2)}}, \quad {\rm for}\ 0<|t|\leq1\,, \\
    C|t|^{-\frac{1}{m_1}}(1+|t|^{-1}|x|)^{-\frac{n(m-4)}{2(m-2)}}, \quad {\rm for}\ |t|\ge 1\,.\\
\end{array}
\right.
\end{equation}
\end{thm}
Note that $I_1$ has the same structure as the oscillatory integral $I(t,x)$ in \cite{KAY} for 
higher order Schr\"odinger equations, when replacing $P^{1/2}(\xi)$ from \eqref{eq:14} by $P(\xi)$. Thus, Theorem \ref{th3.3} is closely related to Theorem 3.1 of \cite{KAY} (when replacing $m_1$ by $m$). This similarity  is also easily seen on the level of the considered evolution equations: The differential operator of our wave-type equation can be factored as
$$
  \partial_{tt}+P(D_x)=[\partial_t+i\sqrt P (D_x)]\,[\partial_t-i\sqrt P (D_x)]\,,
$$
where each squared bracket corresponds to a time-dependent Schr\"odinger equation.

\section{Decay/growth estimates for wave-type equations}\label{S4}

Here we shall apply the Theorems \ref{th3.1}, \ref{th3.3} to establish $L^p-L^q$ estimates for the solution of the
 following higher order  wave-type equation:$$
\left\{
\begin{array}{ll}
 \partial_{tt}u(t,x)+P(-i\nabla)u(t,x)=0,\quad(t,x)\in {\bf R}\times {\bf R}^n,\\
u(0,x)=u_0(x),\;\partial_tu(0,x)=u_1(x),\quad x\in {\bf R}^n.
\end{array}
\right.
$$
As in \eqref{eq:13}, its solution is given by
$$
u(t,x)={\F}^{-1}\cos \left(P^{1/2}(\xi)t\right)\,{\F}u_0
+{\F}^{-1} \frac {\sin \left(P^{1/2}(\xi)t\right)}{P^{1/2}(\xi)}\,{\F}u_1=:U(t,x)+V(t,x).
$$
For any $a\in {\bf R}$ we define the following set of admissible index pairs.
$$
\triangle_a:=\{(p,q);\ \mbox{$(\frac{1}{p},\frac{1}{q})$}\ {\rm lies\
in\ the\ closed\ quadrangle}\ ABCD\},
$$
where $A=(\frac{1}{2},\frac{1}{2})$, $B=(1,\frac{1}{q_a})$, $C=(1,0)$,
and $D=(\frac{1}{q_a'},0)$ for $q_a:=\frac{n}{\mu_a}$,
$\mu_a:=\frac{mn-4n+2a}{2(m-2)}$, and $\frac{1}{q}+\frac{1}{q'}=1$.
Moreover, we denote the Lorentz space by  $L^{p,q}({\bf
R}^n)$ (see p.48 in \cite{G}). \vskip0.4cm

\begin{center}
\begin{pspicture}
(-1,-0.8)(11,5)
\psline{->}(0,0)(9,0) \rput(9.5,0){$\frac{1}{p}$}
\psline{->}(0,0)(0,4.5) \rput(-0.5,4.5){$\frac{1}{q}$}
\psline(-0.1,3)(0.1,3) \rput(-0.5,3){$\frac12$}
\psline(3,-0.1)(3,0.1) \rput(3,-0.5){$\frac12$}
\psline(6,-0.1)(6,0.1) \rput(6,-0.5){$1$}
\pscircle*(3,3){0.1} \rput(3,3.5){$A$} 
\pscircle*(6,2){0.1} \rput(6,2.5){$B$} 
\pscircle*(6,0){0.1} \rput(6.3,0.5){$C$} 
\pscircle*(4,0){0.1} \rput(4.2,0.5){$D$} 
\pscircle*[linecolor=red](5,3){0.1} \rput(5,3.5){$E$} 
\pscircle*[linecolor=red](3,1){0.1} \rput(2.7,1){$F$} 
\psline(3,3)(6,2) 
\psline(6,2)(6,0) 
\psline(6,0)(4,0) 
\psline(4,0)(3,3) 
\psline[linecolor=red,linestyle=dashed](3,3)(5,3) 
\psline[linecolor=red,linestyle=dashed](5,3)(3,1) 
\psline[linecolor=red,linestyle=dashed](3,1)(3,3) 
\psline[linecolor=red,linestyle=dotted](5,3)(6,2) 
\psline[linecolor=red,linestyle=dotted](3,1)(4,0) 
\pscircle*(3,3){0.1}
\psline[linestyle=dotted](3,3)(6,0) 
\end{pspicture}
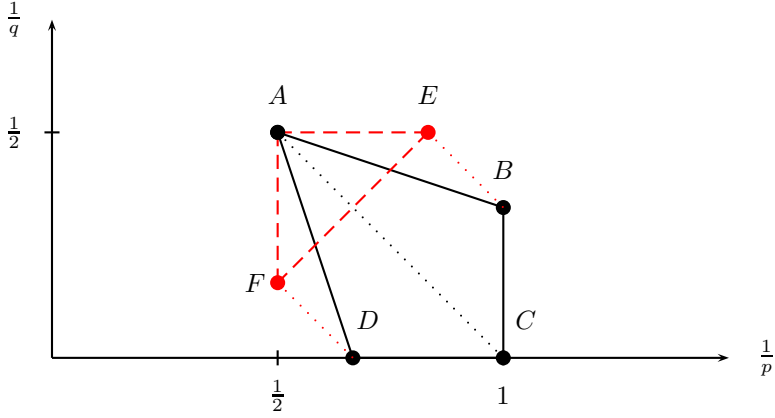
\emph{Solid line: admissible index pairs for Th.\ \ref{th4.1}: $(\frac1p,\frac1q)$ with $(p,q)\in \triangle_m$. Dashed line: admissible index pairs for Th.\ \ref{th4.1a}.
Dotted line: admissible index pairs from Rem.\ \ref{rem:44}(2).
Figure shows the example $m=4$, $n=6$.}
\end{center}

\begin{thm}\label{th4.1}
Assume that the polynomial $P$ satisfies the conditions
{\bf (H1)} and {\bf (H2)}, and let $n\ge m$ (and hence $2\le q_m<\infty$).  Then we have
\begin{equation} \label{eq:41}
\|V(t,\cdot)\|_{L_*^q}\le
\left\{
\begin{array}{ll}
   C|t|^{\frac{n}{m_1}(\frac{1}{q}-\frac{1}{p})+1}\|u_1\|_{L_*^p}, \quad\,\,\,\,\,  {\rm for} \ 0<|t|\leq1, \\
   C|t|^{n|\frac{1}{q}-\frac{1}{p'}|-\frac1m}\|u_1\|_{L_*^p}, \qquad  {\rm for} \ |t|\ge 1, \\
\end{array}
\right.
\end{equation}
where $(p,q)\in\triangle_{m}$, $m_1:=m/2$. Here,
the pair of spaces $(L_*^p, L_*^q)$ has the following meaning:
\begin{equation}\label{eq:41a}
(L_*^p, L_*^q)=
\left\{
\begin{array}{ll}
  ( L^1, L^{q_m,\infty}),\quad  {\rm if} \ \ (p,q)=(1,q_m),\\
(L^{q_m',1}, L^{\infty}),\quad  {\rm if} \ \ (p,q)=
(q_m',\infty),\\
  ( L^p, L^q),\quad  \quad \ \    \quad{\rm otherwise}. \\
\end{array}
\right.
\end{equation}
\end{thm}
\begin{proof} By the assumption \textbf{(H1)}, one has
$$
\left|\frac {\sin \left(P^{1/2}(\xi)t\right)}{P^{1/2}(\xi)}\right|\le
\left\{
\begin{array}{ll}
  |t|, \quad  {\rm for} \ 0<|t|\leq1, \\
  C, \quad {\rm for} \ |t|\ge 1. \\
\end{array}
\right.
$$
Then, the Plancherel theorem gives the result for the index point $A$ :
\begin{equation} \label{eq:42}
\|V(t,\cdot)\|_{L^2}\le
\left\{
\begin{array}{ll}
  |t|\|u_1\|_{L^2}, \quad {\rm for} \ 0<|t|\leq1, \\
   C\|u_1\|_{L^2}, \quad {\rm for} \ |t|\ge 1. \\
\end{array}
\right.
\end{equation}
On the other hand,  by Theorem \ref{th3.1} we have  for each $t\neq0$:
$I_2(t,\cdot)\in L^q({\bf R}^n) $ $\forall$ $q>q_{m}$ and $I_2(t,\cdot)\in
L^{q_m,\infty}({\bf R}^n)$ (the weak $L^{q_m}$ space). Applying the (weak) Young inequality
(see p.22 in \cite{G}) to the second term of \eqref{eq:13} then implies
\begin{equation} \label{eq:43}
\|V(t,\cdot)\|_{L_*^q}\le
\left\{
\begin{array}{ll}
   C|t|^{\frac{n}{m_1}(\frac{1}{q}-1)+1}\|u_1\|_{L^1}, \quad\, {\rm for} \ 0<|t|\leq1, \\
   C|t|^{\frac{n}{q}-\frac1m}\|u_1\|_{L^1}, \quad\quad\quad \,\,\, {\rm for} \ |t|\ge 1. \\
\end{array}
\right.
\end{equation}
This proves the estimate for the points $(1,\frac{1}{q})$ on the edge $\overline{CB}$.
Applying the Marcinkiewicz
interpolation theorem (see p.56 in \cite{G}) to \eqref{eq:42} and \eqref{eq:43} proves
\eqref{eq:41} for the points in the closed triangle $ABC$. By duality, the
estimate for the triangle $ADC$ follows immediately from
the result in the triangle $ABC$ (note that the adjoint operator of $I_2*u_1$ has the same structure).
To include the result for the index point $D$, we remark that $L^{q_m',1}\subset(L^{q_m,\infty})^*$
(cf.\ \cite{Cw}).
This completes the proof of the theorem.
\end{proof}


Next we shall complement this result with a straight forward estimation of
$$
V(t,x)={\F}^{-1} Q(t,\xi)\,{\F}u_1\,,\qquad
Q(t,\xi):=\frac {\sin \left(P^{1/2}(\xi)t\right)}{P^{1/2}(\xi)}\,.
$$
To this end we define the index points $E=(\frac{n+m}{2n},\frac12)$, $F=(\frac12,\frac{n-m}{2n})$.
\begin{thm}\label{th4.1a}
Let the polynomial $P$ satisfy
{\bf (H1)}, and let $n\ge m$.  Then we have
\begin{equation} \label{eq:41b}
\|V(t,\cdot)\|_{L_*^q}\le
\left\{
\begin{array}{ll}
   C|t|^{\frac{n}{m_1}(\frac{1}{q}-\frac{1}{p})+1}\|u_1\|_{L^p}, \quad\,\,  {\rm for} \ 0<|t|\leq1, \\
   C \|u_1\|_{L^p}, \qquad\qquad\qquad\quad  {\rm for} \ |t|\ge 1, \\
\end{array}
\right.
\end{equation}
where $(\frac1p,\frac1q)$ lies in the closed triangle $AEF$ and $m_1:=m/2$.
Here, we denote $L_*^q:=L^{q,\infty}$ if $\frac1q=\frac1p-\frac{m}{2n}$.
And $L_*^q:=L^q$, elsewise.
\end{thm}
\begin{proof} By the assumption \textbf{(H1)}, we have $|Q(t,\xi)|\le C|\xi|^{-m_1}$. And hence,
$Q(t,\cdot)\in L^{\frac{2n}{m},\infty}({\bf R}^n)$.
Since we assumed $u_1\in L^p$ for some $\frac12\le\frac1p\le \frac{n+m}{2n}$, we have
${\F}u_1\in L^{p'}$. And the H\"older inequality for Lorentz spaces (cf.\ \cite{G}) implies
$$
  Q(t,\xi)\,{\F}u_1 \in L^{\tilde p,\infty}\,,\qquad \tilde p:=\frac{p}{p-1+\frac{pm}{2n}}\,.
$$
The Hausdorff-Young inequality for Lorentz spaces (cf.\ \cite{MP}) then yields the result for the edge
$\overline{EF}$ with $\frac1q=\frac1p-\frac{m}{2n}$:
$$
  \|V(t,\cdot)\|_{L^{q,\infty}}\le C \|u_1\|_{L^p}\,,\quad\forall\,t\in{\bf R}\,.
$$
Applying the Marcinkiewicz interpolation theorem (with \eqref{eq:42}) concludes the proof.
\end{proof}

\begin{rem} \label{rem:44}
\begin{enumerate}
\item The short time behavior of $u$ in Th.\ \ref{th4.1} and Th.\ \ref{th4.1a} coincides for the indices in $AEF \cap ABCD$. But for large time, the r.h.s.\ of \eqref{eq:41b} stays uniformly bounded, which is not always the case in \eqref{eq:41}.
\item A Marcinkiewicz interpolation between the edges $\overline{EF}$ and $\overline{BC}$ (plus a duality argument for $\frac1q<\frac1{p'}$) allows to extend the decay/growth estimate on $u$ to the closed hexagon $AEBCDF$. But since this follows exactly the above strategy, we do not give details here.
\item Theorem \ref{th4.1a} actually also holds for $n<m$, but we skipped the statement for notational simplicity.
If $m\in(n,2n)$ one obtains a decay/growth estimate for the index pair $(\frac1p,\frac1q)$ in the closed pentagon described by the five different endpoints: {$(\frac12,\frac12),\,(1,\frac12),\,(1,\frac{n-m_1}{n}),\,(\frac{m_1}{n},0), \,(\frac12,0)$.} And for $m\ge2n$ for the whole index square $\frac12\le \frac1p\le 1$, $0\le\frac1q\le\frac12$.\\
\end{enumerate}
\end{rem}

\noindent
Now we turn to the estimate of $U$:
\begin{thm}\label{th4.2}
Assume that the polynomial $P$ satisfies 
{\bf (H1)} and {\bf (H2)}.  Then we have
\begin{equation} \label{eq:44}
\|U(t,x)\|_{L_*^q}\le
\left\{
\begin{array}{ll}
   C|t|^{\frac{n}{m_1}(\frac{1}{q}-\frac{1}{p})}\|u_0\|_{L_*^p}, \quad\quad\, {\rm for} \ 0<|t|\leq1, \\
   C|t|^{n|\frac{1}{q}-\frac{1}{p'}|-\frac{1}{m_1}}\|u_0\|_{L_*^p}, \quad  {\rm for} \ |t|\ge 1,\\
\end{array}
\right.
\end{equation}
 where $(p,q)\in\triangle_{0}$ and
$(L_*^p, L_*^q)$ is defined in \eqref{eq:41a} (when replacing $q_m$ by $q_0$).
\end{thm}

Using Th. \ref{th3.3}, the proof of Th. \ref{th4.2} is very similar to
Th. \ref{th4.1}. So we omit the details here.

\begin{rem} \label{rem:43}
Let us briefly compare our results to the literature:
While Theorem 2.3 of \cite{P} only yields $L^p-L^{p'}$ estimates for the case $P=1+|\xi|^m$, our Th. \ref{th4.1} provides more
general $L^p-L^q$ estimates. Moreover, our result applies to more general polynomials $P$.
\end{rem}

\section{Appendix: The type of a hypersurface}\label{S:app}

In \S~VIII.3.2 of \cite{St} the \emph{type of a hypersurface}
$S:=\{z=\Phi(\xi);\:\xi\in {\bf R}^n\}\subset {\bf R}^{n+1}$ is defined as follows:
The type $\tilde m(\xi_0)$ of $S$ at $\xi_0$ is the smallest integer $k\ge2$, such that the matrix (or tensor) $\big(\partial^\alpha\Phi(\xi_0)\big)_{|\alpha|=k}$ does \emph{not} vanish. Then, the type of $S$ is 
$\displaystyle \tilde m:=\max_{\xi_0\in {\bf R}^n} \tilde m(\xi_0)$.

\begin{lem} \label{lem:51}
Let $\deg P(\xi)= m\ge4$ and $P(\xi)>0$ on ${\bf R}^n$. Then, the type of $S:=\{z=P^{1/2}(\xi)\}$ satisfies $\tilde m\le m$.
\end{lem}
\begin{proof} 
Assume that there exists a $\xi_0\in{\bf R}^n$ with $\tilde m(\xi_0)>m$. Since $P^{1/2}$ is smooth we have in a small neighborhood around $\xi_0$:
$$
  P^{1/2}(\xi) = P^{1/2}(\xi_0) +(\xi-\xi_0)\cdot \nabla_\xi P^{1/2}(\xi_0)
  + \mathcal O\left(|\xi-\xi_0|^{\tilde m(\xi_0)}\right)\,.
$$
Hence, 
\begin{eqnarray*}
  P(\xi) &=& P(\xi_0) +2(\xi-\xi_0)\cdot \nabla_\xi P^{1/2}(\xi_0)\,\,P^{1/2}(\xi_0)
  + \left[(\xi-\xi_0)\cdot \nabla_\xi P^{1/2}(\xi_0)\right]^2\\
  &&+ \mathcal O\left(|\xi-\xi_0|^{\tilde m(\xi_0)}\right)\,,
\end{eqnarray*}
which contradicts $\deg P(\xi)=m$ with $m\ge4$.
\end{proof}

\begin{rem} \label{rem:52}
\begin{enumerate}
\item If we assume $m=2$ in Lemma \ref{lem:51}, we obtain $\tilde m=2$.
\item In the example $P(\xi)=1+2|\xi|^2+|\xi|^4$ we have $P^{1/2}(\xi)=1+|\xi|^2$,
and hence $\tilde m=2<m$. But in general we can only conclude $\tilde m\le m$ for $m\ge4$.
\end{enumerate}
\end{rem}

\begin{lem} \label{lem:53}
Let the polynomial $P$ on ${\bf R}^n$ satisfy \textbf{(H1)} and \textbf{(H2)}. Then 
$$
{\rm det}\Big{(}\frac{\partial^{2}P^{1/2}(\xi)}{\partial\xi_{i}\partial\xi_{j}}\Big{)}_{n\times
n}\sim c(\frac{\xi}{|\xi|})\,|\xi|^{n(\frac{m}{2}-2)}\qquad \mbox{for $|\xi|$ large,}
$$
where $c$ is a smooth function on the unit sphere of ${\bf R}^n$, bounded away from $0$.
\end{lem}
\begin{proof} 
\noindent
\underline{Step 1:}\\
With $P_m$ denoting the principal part of $P$, we define
$\phi(\xi):=P_m^{1/m}(\xi)$, which is positive for $\xi\ne0$ and homogeneous of degree one.
Now we consider its level-1-hypersurface
$$
  \Sigma:=\{\xi\in{\bf R}^n\,;\, \phi(\xi)=1\} \subset {\bf R}^n\,.
$$
Since $P_m=\phi^m$ is non-degenerate by assumption \textbf{(H2)} (i.e.\ 
${\rm det}\left(\partial_{i}\partial_{j} \phi^m\right) \ne0$ for $\xi\ne0$), Proposition 4.2 from \cite{CMY} implies that $\Sigma$ is strictly convex and of type 2.
Applying again Proposition 4.2 (with $\lambda=m/2$) implies
\begin{equation}\label{eq:51}
{\rm det}\Big{(}\frac{\partial^{2}P_m^{1/2}(\xi)}{\partial\xi_{i}\partial\xi_{j}}\Big{)}_{n\times
n}\neq0 \quad \forall \ \xi\in {\bf R}^n\backslash\{0\}\,.
\end{equation}

\noindent
\underline{Step 2:}\\
Now we decompose
$$
  P^{1/2}(\xi) = P^{1/2}_m(\xi) + a(\xi)\,,
$$
with $a=\frac{P-P_m}{\sqrt P + \sqrt{P_m}}\in S^{\frac{m}{2}-1}$.
Hence, 
$$
  {\rm det}\Big{(}\frac{\partial^{2}P^{1/2}(\xi)}{\partial\xi_{i}\partial\xi_{j}}\Big{)}
  = {\rm det}\Big{(}\frac{\partial^{2}P^{1/2}_m(\xi)}{\partial\xi_{i}\partial\xi_{j}}\Big{)}+Q(\xi)\,,
$$
where the first term on the r.h.s.\ is $\mathcal O\left(|\xi|^{n(\frac{m}{2}-2)}\right)$ for $\xi$ large,
and the second term is of the order $\mathcal O\left(|\xi|^{n(\frac{m}{2}-2)-1}\right)$.
The claim then follows from \eqref{eq:51}.
\end{proof}



\end{document}